\documentclass[11pt]{amsart}

\usepackage{amsmath,amssymb,amsthm,amscd,amsfonts}
\usepackage{amsmath,verbatim,psfrag}
\usepackage[all]{xy}
\usepackage{pinlabel}

\newcommand{\nc}{\newcommand}
\nc{\nt}{\newtheorem}
\nc{\bs}{\bigskip}
\nc{\dmo}{\DeclareMathOperator}

\nt*{main}{Main Theorem}
\nt{theorem}{Theorem}[section]
\nt{proposition}[theorem]{Proposition}
\nt{lemma}[theorem]{Lemma}
\nt{fact}[theorem]{Fact}
\nt{corollary}[theorem]{Corollary}
\nt{conjecture}[theorem]{Conjecture}
\nt{question}[theorem]{Question}
\nt{problem}[theorem]{Problem}

\nc{\E}{\mathcal{E}}
\nc{\SB}{\mathcal{B}}
\nc{\M}{\mathcal{M}}
\nc{\SI}{\mathcal{SI}}
\nc{\D}{D}
\dmo{\Mod}{Mod}
\dmo{\PMod}{PMod}
\dmo{\SMod}{SMod}
\dmo{\SHomeo}{SHomeo}
\dmo{\Homeo}{Homeo}
\dmo{\Sp}{Sp}
\dmo{\SL}{SL}
\dmo{\spl}{\mathfrak{sp}}
\dmo{\SSp}{SSp}
\dmo{\even}{even}

\dmo{\Aut}{Aut}
\dmo{\Sym}{Sym}

\nc{\ia}{\hat\imath}

\nc{\al}{\alpha}
\nc{\be}{\beta}
\nc{\ga}{\gamma}
\nc{\de}{\delta}
\nc{\ep}{\epsilon}

\nc{\Z}{\mathbb{Z}}
\nc{\R}{\mathbb{R}}
\nc{\N}{\mathbb{N}}
\nc{\C}{\mathbb{C}}

\dmo{\B}{B}
\dmo{\Br}{K}
\dmo{\Brun}{Brun}
\dmo{\T}{\mathcal T}
\dmo{\K}{K}
\dmo{\Ker}{Ker}
\dmo{\PB}{PB}
\dmo{\GL}{GL}
\dmo{\Ann}{Ann}
\nc{\BI}{\mathcal{BI}}
\nc{\Push}{\Psi}

\nc{\p}[1]{\medskip\paragraph{{\bf #1}}}
\nc{\margin}[1]{\marginpar{\scriptsize #1}}

\nc{\bl}{ \begin{list}{$\cdot$}{
\setlength{\leftmargin}{.5in}
\setlength{\rightmargin}{.5in}
\setlength{\parsep}{0.5ex plus .2ex minus 0ex}
\setlength{\itemsep}{0.2ex plus 0.2ex minus 0ex}
}
}

\nc{\el}{\end{list}}

\title{The level four braid group$^\dagger$}

\begin{document}
	
\input{epsf.sty}

\author{Tara E. Brendle}

\author{Dan Margalit}

\address{Tara E. Brendle \\ School of Mathematics \& Statistics\\ 15 University Gardens \\ University of Glasgow \\ G12 8QW \\ tara.brendle@glasgow.ac.uk}

\address{Dan Margalit \\ School of Mathematics\\ Georgia Institute of Technology \\ 686 Cherry St. \\ Atlanta, GA 30332 \\  margalit@math.gatech.edu}

\thanks{The first author is supported in part by EPSRC grant EP/J019593/1. The second author supported by the National Science Foundation under Grant No. DMS - 1057874. \\ \hspace*{3ex} $^\dagger$This version is an update of the published version.}



\begin{abstract}
By evaluating the Burau representation at $t=-1$, we obtain a symplectic representation of the braid group.  We define the congruence subgroups of the braid group to be the preimages of the principal congruence subgroups of the symplectic group.  Our main result is that the level four congruence subgroup of the braid group is equal to the group generated by squares of Dehn twists and is also equal to the group generated by squares of pure braids.    We also compute the image of the point pushing subgroup under the symplectic representation.	
\end{abstract}

\maketitle

\vspace*{-2ex}

\section{Introduction}

The integral Burau representation of the braid group is the representation $\rho : \B_n \to \GL_n(\Z)$  obtained by evaluating the (unreduced) Burau representation $\B_n \to \GL_n(\Z[t,t^{-1}])$ at $t=-1$.  The level $m$ congruence subgroup $\B_n[m]$ of $\B_n$ is the kernel of the mod $m$ reduction
\[ \B_n \stackrel{\rho}{\to} \GL_n(\Z) \to \GL_n(\Z/m). \]
As we explain in Section~\ref{sec:burau}, $\rho$ can be regarded as a symplectic representation.   Hence $\rho$ plays the same role for the braid group as the classical symplectic representation does for mapping class groups of closed surfaces.
The representation $\rho$ has connections to many areas of mathematics, such as algebraic geometry, number theory, dynamics, and topology; see, e.g., the work
A'Campo \cite{acampo}, Arnol'd \cite{arnold}, Smythe \cite{smythe}, Band--Boyland \cite{bandboyland}, Cohen--Wu \cite{cohenwu}, Funar--Kohno \cite{FunarKohno}, Gambaudo--Ghys \cite{GambaudoGhys}, Hain \cite{hain}, Khovanov--Seidel \cite{KhovanovSeidel}, Magnus--Peluso \cite{MagnusPeluso},
McMullen \cite{mcmullen}, Morifuji \cite{morifuji}, Mumford  \cite{mumford}, Venkataramana \cite{Venk}, Wajnryb \cite{bw}, and Yu \cite{yu}.

The mapping class group $\Mod(S)$ of a surface $S$ with marked points is the group of homotopy classes of homeomorphisms of $S$ fixing the set of marked points and fixing $\partial S$ pointwise.  Let $\D_n$ denote a closed disk with $n$ marked points in the interior.  We have the following classical fact:
\[ \B_n \cong \Mod(\D_n). \]
As such, it is natural to ask for descriptions of the $\B_n[m]$ that are intrinsic to either braid groups or mapping class groups.  The first result in this direction is due to Arnol'd \cite{arnold} who proved that $\B_n[2]$ is equal to the pure braid group $\PB_n$.  Artin had previously proved that the latter is identified with the subgroup of $\Mod(\D_n)$ generated by Dehn twists.  Denote by $\T_n[m]$ the subgroup of $\Mod(\D_n)$ generated by the $m$th powers of all Dehn twists.  Identifying $\B_n$ with $\Mod(D_n)$, we can summarize the theorems of Arnol'd and Artin as:
\[ 
\B_n[2] = \PB_n = \T_n[1].
\]
Our main theorem gives an analogue for $\B_n[4]$.  Let $\PB_n^2$ be the subgroup of $\PB_n$ generated by the squares of all elements; note that for any group $G$, the group $G^2$ equals the kernel of $G \to H_1(G;\Z/2)$.

\begin{main}
\label{theorem:main}
For $n \geq 1$, we have $\B_{n}[4] = \PB_n^2 = \T_{n}[2]$.
\end{main}

The first equality of our Main Theorem, which is proven in Section~\ref{sec:level four mod 2}, is known in the case $n$ is odd; it appears in the unpublished paper of Yu \cite{yu}.   This equality has a natural interpretation in terms of moduli spaces; see Section~\ref{section:back}.

The second equality of our Main Theorem, proven in Section~\ref{sec:sql}, has a precursor in the case of the mapping class group of a closed, orientable surface of genus $g$: Humphries \cite{Humphries} proved that the level two mapping class group, that is, the kernel of the map $\Mod(S_g) \to \Sp_{2g}(\Z/2)$ given by the action of $\Mod(S_g)$ on $H_1(S_g;\Z/2)$, is equal to the subgroup of $\Mod(S_g)$ generated by all squares of Dehn twists about nonseparating curves.

In Section~\ref{sec:brun} we determine the image under $\rho$ of the subgroup of $\B_n$ consisting of those braids that become trivial when one strand is deleted; see Theorem~\ref{theorem:pointpushing} below.  This subgroup is also known as the point pushing subgroup.

\medskip

\noindent \emph{Related results.} 
Besides the result of Humphries already mentioned, there are various other results about subgroups of $\B_n$ generated by powers of basic elements.  Coxeter \cite[Section 10]{coxeter} showed that the normal closure in $\B_n$ of the $m$th power of any standard generator for $\B_n$ has finite index in $\B_n$ if and only if $1/n + 1/m \leq 1/2$; see also \cite{assion}.  Funar--Kohno \cite[Theorem 1.1]{FunarKohno} proved that the intersection over $m$ of the groups $\T_n[m]$ is trivial.  Humphries \cite[Theorem 1]{Humphries2} gave a complete description of when the group generated by (possibly differing) powers of Artin's (finitely many) generators for $\PB_n$ generates a subgroup of finite index; for instance, the group generated by the squares of Artin's generators has infinite index, in contrast to our Main Theorem.

Next, by evaluating the Burau representation at any $d$th root of unity we obtain an analogue of the integral Burau representation.  Building on work of  McMullen \cite{mcmullen},  Venkataramana \cite{Venk} showed that when $n/2 \geq d \geq 3$ the image of $\B_n$ is arithmetic and is (up to finite index) as large as can be expected.  Deligne--Mostow \cite{DM} previously gave analogues where the image is not arithmetic and Thurston \cite{Thurston} gave an interpretation of their work in terms of moduli spaces of convex polyhedra. 

Finally, there is a more general notion of a congruence subgroup of the braid group.  There is a natural action of $\B_n$ on the free group $F_n$, which can be identified with the fundamental group of the disk with $n$ punctures.  If $H$ is a characteristic subgroup of $F_n$, there is an induced homomorphism $\B_n \to \Aut(F_n/H)$ and the kernel is called a congruence subgroup of $\B_n$.  It is a theorem of Asada \cite{asada} that every finite-index subgroup of $\B_n$ contains such a congruence subgroup; see also \cite{ddh}.  Thurston later gave a more elementary proof; see the exposition by McReynolds \cite{mcr}.

\p{Acknowledgments} We would like to thank Joan Birman, Neil Fullarton, Louis Funar, Lalit Jain, Joseph Rabinoff, Douglas Ulmer, Kirsten Wickelgren, and Mante Zelvyte for helpful conversations.  We would especially like to thank David Benson, Andrew Putman, and Nick Salter who pointed out mistakes in earlier versions and made other useful comments.  We are also grateful to Jordan Ellenberg; a conversation with him on MathOverflow partly inspired the work in the last section.


\section{The Burau representation and a theorem of Arnol'd}
\label{section:back}

In this section we give a description of the integral Burau representation in terms of mapping class groups and also explain the classical result of Arnol'd that $\B_n[2]$ is equal to $\PB_n$.  Then we give a reinterpretation of the first equality of our Main Theorem in terms of moduli spaces of points in $\C$.

\subsection{The Burau representation}\label{sec:burau} Let $\sigma_1,\dots,\sigma_{n-1}$ denote the standard generators for $\B_n$.  The (unreduced) Burau representation is the representation $\B_n \to \GL_n(\Z[t,t^{-1}])$ defined by 
\[ \sigma_i \mapsto I_{i-1} \oplus \left(\begin{smallmatrix} 1-t & \,t \\ 1 & \,0\end{smallmatrix}\right) \oplus I_{n-i-1}. \]
This representation obviously fixes the vector $(1,1,\dots,1)$ and this gives a 1-dimensional summand.   The other summand is called the reduced Burau representation.

\begin{figure}
\labellist
\small\hair 2pt
 \endlabellist
\includegraphics[scale=1.25]{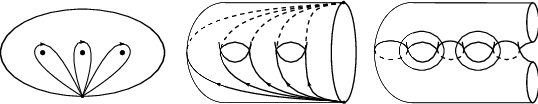}
\caption{\emph{Left to right:} the $\gamma_i$ in $\D_3$, the lifts of the $\gamma_i$ to $X_5$, and the curves in  $X_6$ whose Dehn twists lift the standard generators for $\B_6$}
\label{figure:burau}
\end{figure}

The Burau representation can also be described via topology.  Let $\D_n^\circ$ denote the punctured disk obtained from $\D_n$ by removing the marked points and let $p \in \partial \D_n^\circ$.  Let $Y_n$ denote the universal abelian cover of $\D_n^\circ$, let $t$ denote a generator for the deck transformation group, and let $\tilde p$ denote the full preimage of $p$.  As a $\Z[t,t^{-1}]$-module, $H_1(Y_n,\tilde p;\Z)$ has rank $n$; the generators are represented by path lifts to $Y_n$ of the loops $\gamma_i$ in $\D_n^\circ$ shown in the left-hand side of Figure~\ref{figure:burau} (so the vector $(1,1,\dots,1)$ corresponds to a peripheral loop).  Since the cover $Y_n$ is characteristic, each element of $\B_n$ induces a $t$-equivariant homeomorphism of $Y_n$ and the induced action on $H_1(Y_n,\tilde p;\Z)$ is nothing other than the Burau representation.

\p{The integral Burau representation} As mentioned, the integral Burau representation is the representation $\rho : \B_n \to \GL_n(\Z)$  obtained by evaluating the Burau representation $\B_n \to \GL_n(\Z[t,t^{-1}])$ at $t=-1$.

We can again describe this representation from the topological point of view.  We consider the two-fold branched cover $X_n \to \D_n$ with branch locus equal to the set of marked points.  If $n=2g+1$ then $X_n$ is a compact orientable surface $S_g^1$ of genus $g$ with one boundary component, and if $n=2g+2$ then $X_n$ is a compact orientable surface $S_g^2$ of genus $g$ with two boundary components.  

Again because the (branched) cover $X_n \to \D_n$ is characteristic, each element of $\Mod(\D_n) \cong \B_n$ lifts to a (unique) element of $\Mod(X_n)$ and there is an induced homomorphism $\Mod(\D_n) \to \Mod(X_n)$; denote the image by $\SMod(X_n)$.  It is a special case of a theorem of Birman and Hilden \cite{birmanhilden} that this homomorphism is injective, but we will not use this fact.

Let $\tilde p=\{p_1,p_2\}$ be the preimage in $\partial X_n$ of $p$.  Then $H_1(X_n,\tilde p;\Z) \cong \Z^n$.  Indeed, a basis for $H_1(X_n,\tilde p;\Z)$ is given by the path lifts of the $\gamma_i$; see the middle of Figure~\ref{figure:burau}. We claim that the composition 
\[ \B_n \to \Mod(X_n) \to \Aut(H_1(X_n,\tilde p;\Z)) \subseteq \GL_n(\Z) \]
is again the integral Burau representation.  This can be easily checked by directly computing the action of each of the standard generators for $\B_n$.  Alternatively, one can show that the kernel of the map $H_1(Y_n,\tilde p;\Z) \to H_1(X_n,\tilde p;\Z)$ induced by the natural map $Y_n \to X_n$ is generated by elements of the form $tx+x$; this plus the fact that the lifts of an element of $\Mod(\D_n)$ to $X_n$ and $Y_n$ are compatible gives the claim (the map $Y_n \to X_n$ is not surjective but is a covering map of $Y_n$ onto its image).

We can easily see from the latter description of the integral Burau representation that the reduced integral Burau representation of $\B_{2g+1}$ is symplectic.  Indeed, $H_1(S_g^1,\tilde p;\Z)$ naturally splits as $H_1(S_g^1;\Z) \oplus \Z$ and the first factor carries a symplectic form---the algebraic intersection number---which is preserved by $\B_n$.   The algebraic intersection form on $H_1(S_g^2,\tilde p;\Z)$ is already symplectic, and so the unreduced Burau representation of $\B_{2g+2}$ is symplectic (and reducible).  

Symplectic bases $(\vec x_i,\vec y_i)$ for both cases are shown in Figure~\ref{figure:basis}.  The algebraic intersection number $\hat\imath(\vec x_k,\vec y_k)$ equals 1 for all $k$ and all other algebraic intersections between basis elements are zero.  Throughout, we refer to these bases $\{\vec x_i,\vec y_i\}$ as the standard symplectic bases for $H_1(S_g^1;\Z)$ and $H_1(S_g^2,\tilde p;\Z)$.  For the second case, notice that each boundary component represents the basis element $\vec y_{g+1}$ and so the integral Burau representation can be regarded as a representation
\[ \rho : \B_n \to \begin{cases} \Sp_{2g}(\Z) & n=2g+1 \\ (\Sp_{2g+2}(\Z))_{\vec y_{g+1}} & n = 2g+2 \end{cases} \]
(in the case $n=2g+1$ we have dropped the trivial summand).

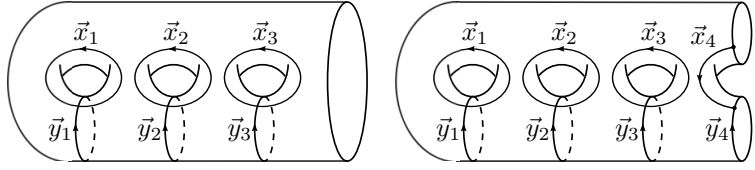
\begin{figure}
\labellist
\small\hair 2pt
 \pinlabel {$\vec y_1$} [ ] at 16 10
 \pinlabel {$\vec y_2$} [ ] at 43 10
 \pinlabel {$\vec y_3$} [ ] at 70 10
 \pinlabel {$\vec y_1$} [ ] at 133 10
 \pinlabel {$\vec y_2$} [ ] at 160 10
 \pinlabel {$\vec y_3$} [ ] at 187 10
 \pinlabel {$\vec y_4$} [ ] at 214 10
 \pinlabel {$\vec x_1$} [ ] at 24 39
 \pinlabel {$\vec x_2$} [ ] at 51 39
 \pinlabel {$\vec x_3$} [ ] at 78 39
 \pinlabel {$\vec x_1$} [ ] at 141 39
 \pinlabel {$\vec x_2$} [ ] at 168 39
 \pinlabel {$\vec x_3$} [ ] at 195 39
 \pinlabel {$\vec x_4$} [ ] at 210 38
\endlabellist
\includegraphics[scale=1.25]{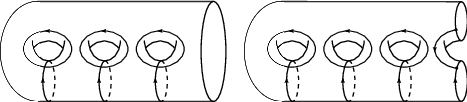}
\caption{The standard symplectic bases for $H_1(S_g^1;\Z)$ and $H_1(S_g^2,\tilde p;\Z)$}
\label{figure:basis}
\end{figure}


\subsection{The pure braid group as a congruence subgroup} We now explain the theorem of Arnol'd that $\PB_n = \B_n[2]$.  There is a canonical basis for $H_1(\D_n^\circ;\Z/2)$ whose elements correspond to the $n$ punctures (these are represented by the $\gamma_i$ above).  Let $H_1(\D_n^\circ;\Z/2)^{\even}$ denote the subspace consisting of elements with an even number of nonzero coordinates in the standard basis.

We would like to define a map $H_1(X_n;\Z/2) \to H_1(\D_n^\circ;\Z/2)$ as follows: given a mod two cycle in $X_n$, we modify it by homotopy so that it avoids the fixed points of $\iota$ and then project to $\D_n^\circ$.  A priori this is not well defined, because homotopies in $X_n$ might push a cycle across a fixed point.  Arnol'd proved that the map is indeed well defined \cite[Lemma 1]{arnold} and injective and that the image is $H_1(\D_n^\circ;\Z/2)^{\even}$ (see also \cite[Lemma 8.12 and footnote on p. 145]{mumford}).  The key point is that a simple closed curve in $X_n$ surrounding a fixed point maps to zero in $H_1(\D_n^\circ;\Z/2)$.

The isomorphism $H_1(\D_n^\circ;\Z/2)^{\even} \to H_1(X_n;\Z/2)$ is $\B_n$-equivariant, and so the elements of $\B_n$ that act trivially on $H_1(X_n;\Z/2)$ are exactly the ones that act trivially on $H_1(\D_n^\circ;\Z/2)^{\even}$.  For $n \geq 3$ these are the braids that fix each marked point of $\D_n$, namely, the pure braids.  Thus $\B_n[2] = \PB_n$.

\subsection{Moduli spaces} The first equality in our Main Theorem has an interpretation in terms of moduli spaces.  Let $\M_{n}^u$ denote the moduli space of configurations of $n$ (unlabeled) points in $\C$.   The double branched cover over such a configuration of points is an (open) hyperelliptic curve.  Such a curve admits a unique hyperelliptic involution, and so we can regard $\M_n^u$ as the moduli space of hyperelliptic curves.  The fundamental group of this moduli space is $\B_n$.

Next, let $\M_n$ denote the configuration space of $n$ labeled points in $\C$.  Because of the identification of $H_1(X_n;\Z/2)$ with $H_1(\D_n^\circ;\Z/2)^{\even}$, the ordering of the points in some configuration of points gives rise to a basis for the mod two homology of the associated hyperelliptic curve, namely, the differences of consecutive points in the configuration.  The fundamental group of $\M_n$ is $\PB_n$ and the forgetful map $\M_n \to \M_n^u$ is the covering map associated to the inclusion $\PB_n \to \B_n$.

Now let $n \geq 3$ and let $m$ be any positive even integer.  Given a point in $\M_n$, we may consider the associated open hyperelliptic curve $X$.  A \emph{hyperelliptic level $m$ marking} of $X$ is a basis for $H_1(X;\Z/m)$ whose mod two reduction is the canonical one given in the previous paragraph.  Let $\M_n[m]$ denote the moduli space of open hyperelliptic curves as above with hyperelliptic level $m$ markings (so $\M_n[2] = \M_n$).  The space $\M_n[m]$ is connected and has fundamental group $\B_n[m]$.  The forgetful map $\M_n[m] \to \M_n$ is the covering map corresponding to the inclusion $\B_n[m] \to \PB_n$.  

Since $\PB_n/\PB_n^2 = \PB_n/\B_n[4]$ is the universal 2-primary abelian quotient of $\PB_n$ we obtain the following corollary, also observed by Yu \cite[Corollary 7.4]{yu} in the case $n$ odd.

\begin{corollary}
\label{cor:universal}
For $n \geq 3$, the covering space $\M_n[4] \to \M_n$ is universal among 2-primary covering spaces of $\M_n$.
\end{corollary}

The space $\M_n[4]$ has an algebro-geometric description as follows:
\[ \textrm{Spec}\ \C [t_i : 1 \leq i \leq n][(t_i-t_j)^{-1},\sqrt{t_i-t_j} : 1 \leq i < j \leq n];\]
this is the so-called K\"ummer cover of $\M_n$.  That these two covering spaces are isomorphic follows, for instance, from Corollary~\ref{cor:universal} and the fact that the deck groups are the same.


\section{Level four versus the mod two kernel}
\label{sec:level four mod 2}

The goal of this section is to prove the following proposition, which is one half of our Main Theorem.  

\begin{proposition}
\label{prop:level four vs kernel}
For any $n$, we have $\B_{n}[4]=\PB_n^2$.
\end{proposition}

For the case of $n=2g+1$, Proposition~\ref{prop:level four vs kernel} follows easily from the well-known facts Theorem~\ref{thm:acampo plus}(1) and Lemma~\ref{lemma:sp}(1) below.  As mentioned, in this case the observation was already made by Yu \cite[Proof of Corollary 7.4]{yu}.  Most of the work in this section is devoted to proving the analogs of these results for the case of $n=2g+2$, namely Theorem~\ref{thm:acampo plus}(2) and Lemma~\ref{lemma:sp}(2).  

Throughout this section, denote the symplectic form on $\Z^{2g}$ by $\hat\imath$ and fix a symplectic basis $\SB_g = \{ \vec x_1, \vec y_1,\dots,\vec x_g,\vec y_g\}$ with $\hat\imath(\vec x_k,\vec y_k)=1$ for all $k$.

\p{Symplectic transvections} The \emph{symplectic transvection} associated to $\vec v \in \Z^{2g}$ is the linear transformation $\tau_{\vec v} : \Z^{2g} \to \Z^{2g}$ given by
\[ \tau_{\vec v}(\vec w) = \vec w + \hat\imath(\vec w,\vec v)\vec v. \]
If $c$ is a simple closed curve in $S_g^1$ that with some choice of orientation represents $\vec v \in H_1(S_g^1;\Z)$, then the image of the Dehn twist $T_c$ in $\Sp_{2g}(\Z)$ is the transvection $\tau_{\vec v}$ (this makes sense because $\tau_{\vec v} = \tau_{-\vec v}$).

The first statement of the next proposition is a slight variation of a classical theorem found, for instance, in the book by Mumford \cite[Proposition A3]{tata1}.  (The modified generating set will make our pictures simpler in the proof of Theorem~\ref{thm:acampo plus}.)

\begin{proposition}
\label{prop:sp trans gens}
Let $g \geq 2$.
\begin{enumerate}
\item The group $\Sp_{2g}(\Z)[2]$ is generated by the $\tau_{\vec v}^2$ with $\vec v$ in
\[ \{\vec x_i\} \cup \{\vec y_j\} \cup \{\vec x_i + \vec x_j \} \cup \{\vec y_i - \vec y_j \} \cup \{\vec x_i - \vec y_j \}. \]
\item The group $\left(\Sp_{2g+2}(\Z)[2]\right)_{\vec y_{g+1}}$ is generated by the $\tau_{\vec v}^2$ with $\vec v$ in
\[ \{\vec x_i \mid i \neq g+1 \} \cup \{\vec y_j\} \cup \{\vec x_i + \vec x_j\mid i,j \neq g+1 \} \cup \{\vec y_i - \vec y_j \} \cup \{\vec x_i - \vec y_j \mid i \neq g+1 \}. \]
\end{enumerate}
\end{proposition}

\begin{proof}

According to Mumford, $\Sp_{2g}(\Z)[2]$ is generated by the $\tau_{\vec v}^2$ with $\vec v$ in
\[  \SB_{g}  \cup \{ \vec u + \vec w \mid \vec u,\vec w \in \SB_{g}, \vec u \neq \vec w \}. \]
For the first statement it suffices to show that we can replace the $\vec y_i+\vec y_j$ and $\vec x_i+\vec y_j$ in this generating set with the corresponding $\vec y_i-\vec y_j$ and $\vec x_i-\vec y_j$.

The first observation is that 
\[ \tau_{\vec x_i}^2 \tau_{\vec x_i + \vec y_i}^2 \tau_{\vec x_i}^{-2} = \tau_{\vec x_i - \vec y_i}^2 \]
and so we may replace the $\vec x_i + \vec y_i$ with the corresponding $\vec x_i - \vec y_i$.  

Next, we note that $\omega_i =  \tau_{\vec x_i}^2 \tau_{\vec y_i}^2 \tau_{\vec x_i - \vec y_i}^2$ negates $\vec x_i$ and $\vec y_i$ while fixing all other elements of $\SB_g$.  It follows that for $i \neq j$ we have
\begin{align*}
\omega_j \tau_{\vec y_i + \vec y_j}^2 \omega_j^{-1} &= \tau_{\vec y_i-\vec y_j}^2 \text{ and} \\
\omega_j \tau_{\vec x_i + \vec y_j}^2 \omega_j^{-1} &= \tau_{\vec x_i-\vec y_j}^2,
\end{align*}
and so we can replace the $\vec y_i + \vec y_j$ and $\vec x_i + \vec y_j$ with $\vec y_i-\vec y_j$ and $\vec x_i-\vec y_j$, as desired.

We proceed to the second statement.  Let $M \in  \left(\Sp_{2g+2}(\Z)[2]\right)_{\vec y_{g+1}}$.  Since $M$ fixes $\vec y_{g+1}$, is the identity modulo two, and preserves $\hat\imath(\vec x_{g+1},\vec y_{g+1})$, we have
\[ M(\vec x_{g+1}) = \left(2c_1\vec x_1 + 2d_1\vec y_1 + \cdots + 2c_g\vec x_g + 2d_g\vec y_g\right) +\vec x_{g+1} + 2d_{g+1} \vec y_{g+1}\]
for some $c_1,d_1,\dots,c_g,d_g,d_{g+1} \in \Z$.  
If for $i < g+1$ we apply the product $\tau_{\vec y_{g+1}-\vec x_i}^{-2c_i}\tau_{\vec x_i}^{2c_i}$ or $\tau_{\vec y_{g+1}-\vec y_i}^{-2d_i}\tau_{\vec y_i}^{2d_i}$, the effect is to eliminate the corresponding term $2c_i\vec x_i$ or $2d_i\vec y_i$ at the expense of changing the coeffidcient $d_{g+1}$.  We thus reduce to the case where $M(\vec x_{g+1})$ lies in $\langle \vec x_{g+1}, \vec y_{g+1} \rangle$ and where $M$ fixes $\vec y_{g+1}$.  By then applying a power of $\tau_{\vec y_{g+1}}^2$ we reduce to the case that $M$ fixes both $\vec x_{g+1}$ and $\vec y_{g+1}$.

Since $M$ preserves the symplectic form and fixes $\langle \vec x_{g+1}, \vec y_{g+1} \rangle$ it follows that $M$ preserves $\langle \vec x_{1}, \vec y_{1}, \dots, \vec x_{g}, \vec y_{g} \rangle$, that is, $M$ decomposes as a direct sum $M_g \oplus I_2$, where $M_g$ is the action on $\langle \vec x_{1}, \vec y_{1}, \dots, \vec x_{g}, \vec y_{g} \rangle$.  Since the image of the generating set from the first statement under the inclusion $\Sp_{2g}(\Z) \to \left(\Sp_{2g+2}(\Z)[2]\right)_{\vec y_{g+1}}$ is contained in the generating set from the second statement, the proposition follows.
\end{proof}

Below we show that $(\Z/2)^{2g+1 \choose 2}$ is a quotient of $\Sp_{2g}(\Z)[2]$ from which it follows that the generating set for $\Sp_{2g}(\Z)$ in Proposition~\ref{prop:sp trans gens} is minimal.  See the paper of Church and Putman \cite{CP} for minimal generating sets for other congruence subgroups of $\Sp_{2g}(\Z)$.

The first statement of the following theorem is due to A'Campo \cite[Th\'eor\`eme 1]{acampo} (see also Mumford \cite[Lemma 8.12]{mumford} and Wajnryb \cite[Theorem 1]{bw}).

\begin{figure}[htb]
\labellist
\small\hair 2pt
 \pinlabel {$c_{\vec y_2}$} [ ] at 41 63
 \pinlabel {$c_{\vec x_1}$} [ ] at 30 95
 \pinlabel {$c_{\vec x_4-\vec y_4}$} [ ] at 79 93
 \pinlabel {$c_{\vec x_2 + \vec x_4}$} [ ] at 179 64
 \pinlabel {$c_{\vec y_2 - \vec y_4}$} [ ] at 45 41
 \pinlabel {$c_{\vec x_2 - \vec y_4}$} [ ] at 177 11
\endlabellist
\centering
\includegraphics{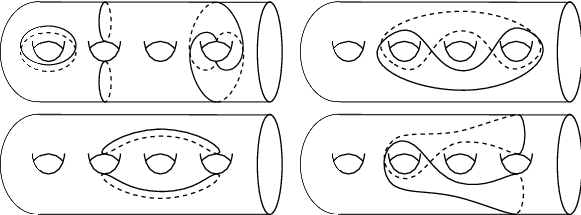}
\caption{The curves used in the proof of Theorem~\ref{thm:acampo plus}}
\label{figure:se}
\end{figure}
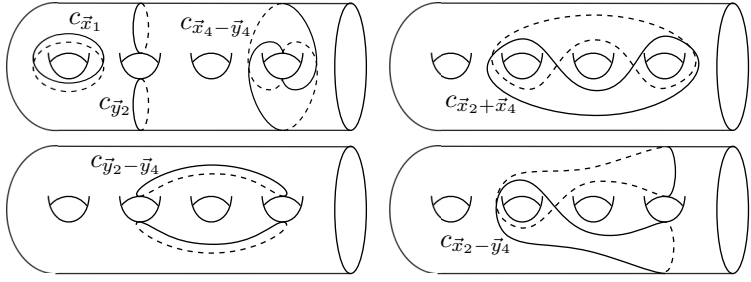

\begin{theorem}
\label{thm:acampo plus}
Let $g \geq 2$.
\begin{enumerate}
\item The restriction $\rho : \PB_{2g+1} \to \Sp_{2g}(\Z)[2]$ is surjective.
\item The restriction $\rho : \PB_{2g+2} \to \left(\Sp_{2g+2}(\Z)[2]\right)_{\vec y_{g+1}}$ is surjective.
\end{enumerate}
\end{theorem}

\begin{proof}[Proof of Theorem~\ref{thm:acampo plus}]

It suffices to realize each generator from parts (1) and (2) of Proposition~\ref{prop:sp trans gens}.  For each transvection $\tau_{\vec v}$ as in the proposition we can find a simple closed curve $c_{\vec v}$ lying in $S_g^1$ or $S_{g}^2$ accordingly and with the following properties:
\begin{enumerate}
\item for some choice of orientation of $c_{\vec v}$, we have $[\vec c_{\vec v}] = \vec v$ and
\item $\iota(c_{\vec v}) \cap c_{\vec v} = \emptyset$.
\end{enumerate}
The required curves are shown in Figure~\ref{figure:se} (for the second statement of the theorem we should imagine $S_g^2$ as lying inside $S_{g+1}^1$ and check that the required curves avoid $S_{g+1} \setminus S_g^2$).  It follows from the second condition that $T_{c_{\vec v}}T_{\iota(c_{\vec v})}$ lies in $\SMod(S_g^1)$ or $\SMod(S_g^2)$, and hence corresponds an element of the appropriate pure braid group ($\PB_{2g+1}$ or $\PB_{2g+2}$).  By the first condition, the image of this product in the appropriate symplectic group is $\tau_{\vec v}^2$, as desired.
\end{proof}

\p{The symplectic Lie algebra} Let $J$ denote the $2g \times 2g$ matrix associated to the symplectic form on $\Z^{2g}$ and let $j$ denote the mod 2 reduction.  Just as $\Sp_{2g}(\Z)$ is the group of integral matrices that satisfy $MJ=JM^T$, the group $\spl_{2g}(\Z/2)$ is the additive group of $2g \times 2g$ matrices $m$ with entries in $\Z/2$ and with $mj=jm^T$.  If we reorder the symplectic basis for $\Z^{2g}$ as $(\vec x_1,\dots,\vec x_g,\vec y_g,\dots,\vec y_1)$, then $\spl_{2g}(\Z/2)$ is the set of matrices over $\Z/2$ that are persymmetric (symmetric along the anti-diagonal).  To put it yet another way, these are the matrices where, for each $\vec v$ and $\vec w$ in the standard basis, the $\vec v \vec w$-entry is equal to the $\vec w^\star \vec v^\star$-entry (really these are the entries corresponding to the mod two reductions of those vectors).  

For $\vec v, \vec w \in \SB_g$, let $m_{\vec v \vec w}$ be the element of $\spl_{2g}(\Z/2)$ obtained from the zero matrix by replacing the $\vec v \vec w$- and $\vec w^\star \vec v^\star$-entries with 1 (if $\vec v = \vec w^\star$ this matrix has a single nonzero entry).
From the definition, we see that  $m_{\vec v \vec w} = m_{\vec w^\star \vec v^\star}$.  Clearly the $m_{\vec v \vec w}$ generate the abelian group $\spl_{2g}(\Z/2)$.

We remark that there is an isomorphism $\spl_{2g}(\Z/2) \to S^2((\Z/2)^{2g})$ given by $m_{\vec v \vec w} \mapsto \vec v \vec w^\star$.  From this or any of the other descriptions of $\spl_{2g}(\Z/2)$, we can easily check that $\spl_{2g}(\Z/2)$ is isomorphic to $(\Z/2)^{2g+1 \choose 2}$.

\p{An abelian quotient of the level 2 symplectic group}  Newman--Smart \cite[Theorem 7]{newmansmart} proved that there is a surjective homomorphism
\[
\psi : \Sp_{2g}(\Z)[2] \to \spl_{2g}(\Z/2)
\]
given by
\[
I_{2g}+2A \mapsto A \mod 2.
\]
Evidently, the kernel of $\psi$ is $\Sp_{2g}(\Z)[4]$.  So there is a short exact sequence
\[ 1 \to \Sp_{2g}(\Z)[4] \to \Sp_{2g}(\Z)[2] \stackrel{\psi}{\to} \spl_{2g}(\Z/2) \to 1. \]

We will need to describe the image of $\left(\Sp_{2g+2}(\Z)[2]\right)_{\vec y_{g+1}}$ under $\psi$.  For any $\vec v \in (\Z/2)^{2g}$, set
\[ \Ann(\vec v) = \{m \in \spl_{2g+2}(\Z/2) \mid m(\vec v) = 0 \}. \]
It is straightforward to check that the image of $\left(\Sp_{2g+2}(\Z)[2]\right)_{\vec y_{g+1}}$ under $\psi$ lies in $\Ann(\vec y_{g+1})$ and that $\Ann(\vec y_{g+1})$ is generated by
\[ \{ m_{\vec v \vec w} \mid \vec v,\vec w \in \{\vec x_1,\vec y_1, \dots, \vec x_{g+1},\vec y_{g+1}\},\  \vec v \neq \vec x_{g+1}, \ \vec w \neq \vec y_{g+1} \}.\]
In particular, $\Ann(\vec y_{g+1})$ is isomorphic to $(\Z/2)^{2g+2 \choose 2}$.

\begin{lemma}
\label{lemma:sp}
Let $g \geq 0$.   We have:
\begin{enumerate}
 \item The map $\psi : \Sp_{2g}(\Z)[2] \to \spl_{2g}(\Z/2)$ is surjective, and 
\item The map $\psi : \left(\Sp_{2g+2}(\Z)[2]\right)_{\vec y_{g+1}} \to \Ann(\vec y_{g+1})$ is surjective.
\end{enumerate}
\end{lemma}

\begin{proof}

As we already said, the first statement is well known.  However, we give a proof for completeness and to establish some notation needed for the second case.
Specifically, for any choice of $\vec v$ and $\vec w$ in the standard basis $\{\vec x_1,\vec y_1,\dots,\vec x_g,\vec y_g\}$, we would like to define an element $M_{\vec v \vec w}$ of $\Sp_{2g}(\Z)[2]$ whose image is the matrix $m_{\vec v \vec w}$ given above.  Let $N_{\vec v}$ denote the matrix obtained from the identity by negating the $\vec v \vec v$- and $\vec v^\star \vec v^\star$-entries.  We set

\[
M_{\vec v \vec w} = 
\begin{cases}
N_{\vec v} & \vec v = \vec w \\
\tau_{\vec v}^2 & \vec v = \vec w^\star \\
\tau_{\vec w^\star+\vec v}^{-2}\tau_{\vec w^\star}^2 \tau_{\vec v}^2 & \text{otherwise}. 
\end{cases}
\]

It is straightforward to check that the image of each $N_{\vec v \vec w}$ is the desired matrix $m_{\vec v \vec w}$.  Since the $m_{\vec v \vec w}$ generate $\spl_{2g}(\Z/2)$, the first statement follows.

We already said that the $m_{\vec v \vec w}$ with $\vec v \neq \vec x_{g+1}$ and $\vec w \neq \vec y_{g+1}$ generate $\Ann(\vec y_{g+1})$ and so for the second statement it suffices to note that the corresponding matrices $M_{\vec v \vec w}$ lie in $(\Sp_{2g+2}(\Z)[2])_{\vec y_{g+1}}$, which follows immediately from the definitions.
\end{proof}

\begin{proof}[Proof of Proposition~\ref{prop:level four vs kernel}]

First we treat the case of $n=2g+1$.  Any map from a group to a 2-primary abelian group factors through its universal $\Z/2$-vector space quotient, and so there is a commutative diagram

\[ 
\xymatrix{
\PB_{2g+1} \ar[r]^{\rho} \ar[dr]_{\alpha} &\Sp_{2g}(\Z)[2] \ar[r]^{\hspace*{-7ex}\psi} & \spl_{2g}(\Z/2) \cong  (\Z/2)^{2g+1 \choose 2} \\
& (\Z/2)^{2g + 1 \choose 2} \ar[ur]_\beta &
}
\]

Since $\psi \circ \rho$ is surjective (Theorem~\ref{thm:acampo plus}(1) and Lemma~\ref{lemma:sp}(1)), it follows that $\beta$ is an isomorphism.  Thus $\ker (\psi \circ \rho)  = \ker \alpha$.  We already said that $\ker \alpha = \PB_{2g+1}^2$.  Since $\ker \psi = \Sp_{2g}(\Z)[4]$ and $\B_{2g+1}[4] \subseteq \B_{2g+1}[2] = \PB_{2g+1}$ we have $\ker (\psi \circ \rho) = \B_{2g+1}[4]$.  This completes the proof in the case $n$ odd.

For $n=2g+2$ even the proof is the same except that we use Theorem~\ref{thm:acampo plus}(2), Lemma~\ref{lemma:sp}(2), and the fact that  $\Ann(\vec y_{g+1}) \cong (\Z/2)^{2g+2 \choose 2}$.
\end{proof}


\section{Squares of twists versus the mod two kernel}
\label{sec:sql}

Combined with Proposition~\ref{prop:level four vs kernel} the following proposition gives the Main Theorem.  

\begin{proposition}
\label{prop:squares versus kernel}
For any $n$ we have $\T_{n}[2] = \PB_n^2$.
\end{proposition}

In order to prove Proposition~\ref{prop:squares versus kernel}, we need a certain relation amongst Dehn twists in $\D_3$.  The configuration of curves involved in this relation is reminiscent of the lantern relation and most of the twists involved are squared; hence we refer to it as the {\it squared lantern relation}.  

\begin{proposition}[The squared lantern relation]
\label{proposition:squaredlantern}
Let $a$, $b$, $c$, $d$, and $e$ be the curves in $\D_3$ shown in Figure~\ref{figure:sql}.   The following relation holds in $\Mod(\D_3)$:
\[ [T_a,T_b]  = T_a^2 T_d^2 T_c^2 T_e^{-2}.  \]
\end{proposition}

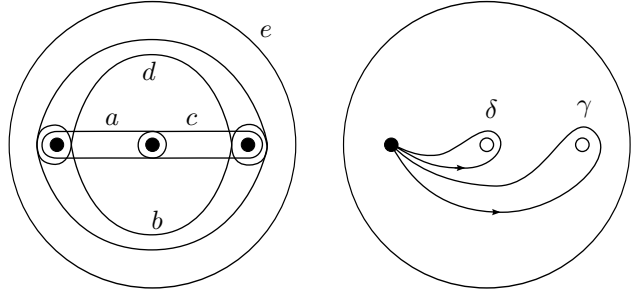
\begin{figure}[h!]
\labellist
\small\hair 2pt
 \pinlabel {$a$} [ ] at 39 64
 \pinlabel {$b$} [ ] at 56 26
 \pinlabel {$c$} [ ] at 69 64
 \pinlabel {$d$} [ ] at 53 82
 \pinlabel {$e$} [ ] at 97 97
 \pinlabel {$\gamma$} [ ] at 217 68
 \pinlabel {$\delta$} [ ] at 182 68
\endlabellist
\includegraphics{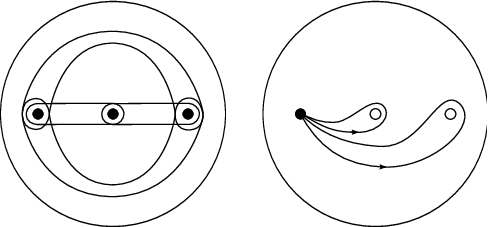}
\caption{The curves in the squared lantern relation and the loops used in the proof}
\label{figure:sql}
\end{figure}

Let $\sigma_1,\dots,\sigma_{n-1}$ denote the standard generators for the braid group $\B_n$.  For $1 \leq i < j \leq n$ let $a_{ij} = \omega^{-1} \sigma_i^2 \omega$, where $\omega = \sigma_{i+1} \cdots \sigma_{j-1}$ (elements of $\B_n$ are composed left to right).  Artin proved that the $a_{ij}$ generate $\PB_n$.

Each $a_{ij}$ is equal to a Dehn twist about a curve $c_{ij}$ in $\D_n$ surrounding two marked points.  If we place the marked points of $\D_n$ in a horizontal line, and if we choose the $\sigma_i$ to be right-handed half-twists, then $c_{ij}$ is the boundary of a regular neighborhood of an arc whose interior lies below the line and connects the $i$th marked point to the $j$th.  

  We can reinterpret the squared lantern relation in terms of Artin's generators as follows:
\[   [a_{12},a_{13}] = a_{12}^2 (a_{12}^{-1}a_{13}^2a_{12})a_{23}^2 (a_{13}a_{12}a_{23})^{-2}  \]
(in the braid group elements are multiplied left to right).

\p{Push maps} While Proposition~\ref{proposition:squaredlantern} can be verified using any of the standard solutions to the word problem for either the mapping class group, the braid group, or the pure braid group, we will give here a conceptual proof.  The ideas we develop here will also be used in the next section.

Choose one marked point of $\D_n$, call it $p$, and delete the other $n-1$ marked points from $\D_n$.  Denote the resulting disk with $n-1$ punctures and one marked point by $\D_n'$.  There is a \emph{push map}:
\[ \Push : \pi_1(\D_n',p) \to \PMod(\D_n) \cong \PB_n\]
defined as follows.  Given $\gamma \in \pi_1(D_n',p)$, we choose an isotopy of $p$ that pushes $p$ along $\gamma$ and we extend this to an isotopy of $\D_n'$.  At the end of the isotopy there is an induced homeomorphism of $\D_n'$, hence $\D_n$, whose homotopy class is $\Push(\gamma)$ (see \cite[Section 4.2]{primer} for details).  

If $\gamma$ has a simple representative $w$ with regular neighborhood $A$ in $\D_n'$, then $\Push(\gamma)$ is equal to the product $T_\ell T_r^{-1}$, where $\ell$ and $r$ denote the components of $\partial A$ lying to the left and right of $\gamma$, respectively (see \cite[Fact 4.7]{primer}).  It is sometimes the case that one of $\ell$ or $r$ is inessential, in which case we can omit the corresponding trivial Dehn twist.   Since products in $\pi_1(\D_n',p)$ are usually written left to right, the map $\Push$ is an antihomomorphism. 

\begin{proof}[Proof of Proposition~\ref{proposition:squaredlantern}]

Choose the marked point $p$ as in the right-hand side of Figure~\ref{figure:sql}.  As above there is a map $\Psi:\pi_1(\D_3',p) \to \PMod(\D_3)$.  Let $\gamma$ and $\delta$ be the two elements of $\pi_1(\D_3',p)$ indicated in the same figure; these generate the free group $\pi_1(\D_3',p)\cong F_2$.  As above we have
\[ \Push(\gamma) = T_b^{-1} , \ \  \Push(\delta) = T_a^{-1} , \ \text{ and} \ \ \Push(\gamma\delta) = T_c T_e^{-1}. 
\]
In the free group $\pi_1(\D_3',p)$, we can write
\[
[\gamma,\delta] = (\gamma\delta)^2 (\gamma\delta)^{-1}\gamma^{-2}(\gamma\delta)\delta^{-2}.
\]
Applying the antihomomorphism $\Push$ to the left-hand side, we obtain the commutator $[T_a, T_b]$.
Applying $\Push$ to the right-hand side and using the above descriptions of $\Push(\gamma)$, $\Push(\delta)$, and $\Push(\gamma\delta)$ in terms of Dehn twists (and remembering that $\Push$ is an antihomomorphism), we obtain
\begin{eqnarray*}
\Push(\delta)^{-2} \Push(\gamma\delta) \Push(\gamma)^{-2} \Push(\gamma\delta)^{-1} \Push(\gamma\delta)^2 
= T_a^2 (T_c T_e^{-1}) T_b^2 (T_c T_e^{-1})^{-1} (T_c T_e^{-1})^2
\end{eqnarray*}
Using now the formula $fT_bf^{-1} = T_{f(b)}$, the fact that $T_c T_e^{-1}(b) = d$, and the fact that $T_c$ and $T_e$ commute we see that the right-hand side is equal to $T_a^2 T_d^2 T_c^2 T_e^{-2}$.  The lemma follows.
\end{proof}

\begin{proposition}
\label{prop:comm}
For $n \geq 3$, the commutator subgroup $\PB_n'$ of $\PB_n$ is normally generated in $\B_n$ by the single element $[a_{12},a_{13}]$.
\end{proposition}

\begin{proof}

First, the commutator subgroup of any group is normally generated in that group by the commutators of the generators.  Thus $\PB_n'$ is normally generated in $\PB_n$ by all of the commutators $[a_{ij},a_{k\ell}]$.

Next, $\Mod(\D_n)$ acts on the set of ordered pairs of distinct curves $(c_{ij},c_{k\ell})$ with three orbits, corresponding to whether the curves have geometric intersection number equal to 0, 2, or 4.  These orbits are represented by the pairs $(c_{12},c_{34})$, $(c_{12},c_{13})$, and $(c_{13},c_{24})$, respectively.  It follows that the action of $\B_n$ on the set of ordered pairs of Artin generators $a_{ij}$ has three orbits, represented by $(a_{12},a_{34})$, $(a_{12},a_{13})$, and $(a_{12},a_{34})$.  As the commutator $[a_{12},a_{34}]$ is trivial, it follows that $\PB_n'$ is normally generated in $\B_n$ by $[a_{12}, a_{23}]$ and $[a_{13}, a_{24}]$.

We have the following relation in $\PB_n$:
\[  [a_{13},a_{24}]= (a_{13}a_{23}^{-1})[a_{23},a_{24}](a_{13}a_{23}^{-1})^{-1} \ a_{23}^{-1}[a_{24},a_{23}]a_{23}  \]
(this relation is obtained by expanding the well-known relator $[a_{23}a_{13}a_{23}^{-1},a_{24}]$ for $\PB_n$ via the Witt--Hall relation $[xy,z]=x[y,z]x^{-1}[x,z]$).  By the previous paragraph, this relation equates $[a_{13},a_{24}]$ with a product of two conjugates in $\B_n$ of $[a_{12}, a_{13}]$.  This completes the proof.
\end{proof}


\p{Forgetful maps} For any $n$ and any $0\leq k \leq n$ there are $n \choose k$ forgetful maps $\PB_n \to \PB_k$ obtained by forgetting $n-k$ strands.  The various forgetful maps $\PB_n \to \PB_2 \cong \Z$ are the coordinates of a surjective homomorphism $\PB_n \to \Z^{n \choose 2}$ which is in fact the abelianization of $\PB_n$.  At the other extreme, the kernel of any forgetful map $\PB_n \to \PB_{n-1}$ corresponds to the image of a push map, so there is a short exact sequence:
\[ 1 \to \pi_1(\D_n',p) \to \PB_n \to \PB_{n-1}\]
(this is a special case of the so-called Birman exact sequence).  This exact sequence has an obvious splitting.

\begin{proof}[Proof of Proposition~\ref{prop:squares versus kernel}]

Since $\PB_n^2$ equals the kernel of the mod two abelianization of $\PB_n$, it follows that $\PB_n^2$ is the preimage of $2\Z^{n \choose 2}$ under the abelianization map
\[ \alpha : \PB_n \to \Z^{n \choose 2}. \]
As above, the $n \choose 2$ coordinates of $\alpha$ are given by the various forgetful maps $\PB_{n} \to \PB_2 \cong \Z$.  The group $\PB_2$ is generated by a Dehn twist.  It follows that $\alpha(\T_{n}[2])$ is precisely $2\Z^{n \choose 2}$, and so $\T_n[2]$ has the same image under $\alpha$ as $\PB_n^2$.  It remains to show that $\T_n[2]$, like $\PB_n^2$, contains the kernel of $\alpha$, namely, the commutator subgroup $\PB_n'$ of $\PB_n$.

By Proposition~\ref{proposition:squaredlantern}, $\T_{n}[2]$ contains the commutator $[a_{12}, a_{13}]$ (see the discussion after the statement).  As $\T_{n}[2]$ is normal in $\B_{n}$ it then follows from Proposition~\ref{prop:comm} that $\T_n[2]$ contains $\PB_n'$, as desired.
\end{proof}

We pause to record the following corollary of the Main Theorem.  In the statement, $\BI_n$ is the kernel of the integral Burau representation of $\B_n$.  

\begin{corollary}
\label{cor:forgettingBI}
Let $\BI_n \leqslant H \leqslant \B_n[4]$.  For any $1 \leq k < n$, the image of $H$ under any forgetful map $\PB_n \to \PB_k$ is $\B_k[4]$.
\end{corollary}

\begin{proof}

Let $F: \PB_n \to \PB_k$ be a forgetful map.  Clearly $F$ preserves squares of Dehn twists, so we have $F(\T_n[2]) \subseteq \T_k[2]$.  Hence by the Main Theorem we have that $F(\B_n[4]) \subseteq \B_k[4]$, and in particular $F(H) \subseteq \B_k[4]$.

For the other containment, let $T_c^2 \in \B_k$.  By the Main Theorem, squares of Dehn twists generate $\B_k[4]$ and so it suffices to show that $T_c^2$ lies in the image of $F$.  We can choose a curve $\tilde c \subseteq \D_n$ so that $\tilde c$ contains an odd number of marked points and so that $\tilde c$ maps to $c$ under the forgetful map $\D_n \to \D_k$ (if $c$ does not already surround an odd number of points, we ``remember'' one marked point inside $c$).  Then $F(T_{\tilde c}^2)=T_c^2$.  Moreover, $T_{\tilde c}^2$ lies in $\BI_n$ as its lift to $\SMod(S_g^1)$ or $\SMod(S_g^2)$ is a Dehn twist about a separating curve.  Thus, $T_{\tilde c}^2$ lies in $H$ by assumption.  The corollary follows.
\end{proof}


\section{Burau images of Point pushing subgroups}
\label{sec:brun}

Denote the $n$ marked points of $\D_n$ by $p_1,\dots,p_n$.  As in Section~\ref{sec:sql}, for each $1 \leq i \leq n$ there is a point pushing subgroup $\pi_1(\D_n',p_i) \subseteq \Mod(\D_n)$.  For any $1 \leq k \leq n$ we define $\Br_{n,k}$ to be the subgroup of $\PB_n$ corresponding to the intersection
\[ \pi_1(\D_n',p_1) \cap \cdots \cap \pi_1(\D_n',p_k). \]
The group $\Br_{n,1}$ is the {\it point pushing subgroup} of $\PB_n$, and $\Br_{n,n}$ is the {\it Brunnian subgroup} $\Brun_n$  of $\PB_n$, that is, the subgroup consisting of the braids that become trivial when any one strand is deleted.  In this section we prove the following proposition; the first statement of which appears in the paper of Yu \cite[Theorem 7.3(iii)]{yu}.

\begin{theorem}  
\label{theorem:pointpushing}
Let $g \geq 2$.  
\begin{enumerate}
\item The group $\rho(\Br_{2g+1,1})$ contains $\Sp_{2g}(\Z)[4]$ and
\[ 
 \rho(\Br_{2g+1,1})/\Sp_{2g}(\Z)[4] \cong (\Z/2)^{2g}.
\]
\item The group $\rho(\Br_{2g+2,1})$ contains $(\Sp_{2g}(\Z)[4])_{\vec y_{g+1}}$ and
\[ 
\hspace*{-3ex}  \rho(\Br_{2g+2,1})/(\Sp_{2g}(\Z)[4])_{\vec y_{g+1}} \cong (\Z/2)^{2g+1}.
\]
\end{enumerate}
\end{theorem}

One can use Theorem~\ref{theorem:pointpushing} to compute the indices of $\rho(\Br_{2g+1,1})$ and $\rho(\Br_{2g+2,1})$ in $\Sp_{2g}(\Z)[2]$ and $(\Sp_{2g}(\Z)[2])_{\vec y_{g+1}}$, respectively, using the facts that
\[ [\Sp_{2g}(\Z)[2]:\Sp_{2g}(\Z)[4]]=2^{2g+1 \choose 2}\] and
\[ [(\Sp_{2g}(\Z)[2])_{\vec y_{g+1}}:(\Sp_{2g}(\Z)[4])_{\vec y_{g+1}}]=2^{2g+2 \choose 2}.\]
For example:
\[ [\Sp_{2g}(\Z)[2]:\rho(\Br_{2g+1,1})] = 2^{g(2g-1)}. \]
We will require the following theorem.  The first statement is due to Mennicke \cite[Section 10]{Mennicke} and the second statement follows easily from the first statement and the same type of considerations as in the proof of Theorem~\ref{thm:acampo plus}.

\begin{theorem}
\label{theorem:mennicke}
Let $g \geq 2$ and let $m \geq 2$.
\begin{enumerate}
\item $\Sp_{2g}(\Z)[m]$ is generated by
\[
\{\tau_{\vec v}^m \mid \vec v \in \Z^{2g} \text{ primitive}\}.
\]
\item $\left(\Sp_{2g+2}(\Z)[m]\right)_{\vec y_{g+1}}$ is generated by
\[
\{\tau_{\vec v}^m \mid \vec v \in \Z^{2g+2} \text{ primitive, }  \hspace*{0ex}  \hat\imath(\vec v,\vec y_{g+1})=0\}.
\]
\end{enumerate}
\end{theorem}

Say that a simple closed curve $c$ in $S_g^1$ is \emph{pre-symmetric} if $\iota(c) \cap c = \emptyset$.  

\begin{proposition}
\label{proposition:presymmetricrep}
If $\vec v \in H_1(X_n;\Z)$ is primitive then it is represented by a pre-symmetric, oriented simple closed curve. 
\end{proposition}

\begin{proof}

We first treat the case $n=2g+1$, in which case $X_n \cong S_g^1$.  It follows from the description of Arnol'd's isomorphism between $H_1(S_g^1;\Z/2)$ and $H_1(\D_{2g+1}^\circ;\Z/2)^{\even}$ that there is a pre-symmetric representative $c'$ in $S_g^1$ of the mod two reduction of $\vec v$.  Indeed, any class in $H_1(\D_{2g+1}^\circ;\Z/2)^{\even}$ is represented by a simple closed curve surrounding an even number of marked points, and the preimage of such a curve has two components, either of which is the desired $c'$.  

Let $\vec v'$ denote the class of $c'$ in $H_1(S_g^1;\Z)$.  The group $\Sp_{2g}(\Z)[2]$ acts transitively on the representatives of a given class in $H_1(S_g^1;\Z/2)$ (see, e.g. \cite[Corollary 3.11]{bmp}) and so there is an $M \in \Sp_{2g}(\Z)[2]$ with $M(\vec v') = \vec v$.  By Theorem~\ref{thm:acampo plus}(1), there is a $b \in \PB_{2g+1}$ with $\rho(b) = M$.  If $\tilde b$ is the corresponding element of $\SMod(S_g^1)$, then $\tilde b(c')$ is the desired representative.  

The case of $n=2g+2$ is almost exactly the same.  The main difference is that we must choose $M$ to lie in $(\Sp_{2g+2}(\Z)[2])_{\vec y_{g+1}}$ (the existence of such an $M$ follows by applying the same statement as before \cite[Corollary 3.11]{bmp} to the pair $\vec v, \vec y_{g+1}$).  We can then apply Theorem~\ref{thm:acampo plus}(2) to complete the proof.
\end{proof}

One can prove Proposition~\ref{proposition:presymmetricrep} without Theorem~\ref{thm:acampo plus} by applying a hyperelliptic version of the Euclidean algorithm for simple closed curves due to Meeks and Patrusky (see \cite[Proposition 6.2]{primer}).

\p{Symmetric homology classes} We remark that the pre-symmetric representative of $\vec v$ given by Proposition~\ref{proposition:presymmetricrep} is homotopic to a symmetric one (that is, one fixed by the hyperelliptic involution) if and only if the corresponding element of $H_1(\D_{2g+1}^\circ;\Z/2)^{\even}$ has exactly two nonzero entries in the standard basis, that is, if and only if the corresponding simple closed curve in $\D_n$ surrounds exactly two marked points.  In particular, the existence of a symmetric representative of $\vec v$ is completely determined by the mod two reduction of $\vec v$.  This was observed by A'Campo \cite[Th\'eor\`eme 3]{acampo}; see also Wajnryb \cite[Theorem 2 and Corollary]{bw}.

\begin{proof}[Proof of Theorem~\ref{theorem:pointpushing}]

We begin with the first statement, which concerns the odd-stranded braid groups $\PB_{2g+1}$ with $g \geq 2$.  The first goal is to prove that $\rho(\Br_{2g+1,1})$ contains $\Sp_{2g}(\Z)[4]$.  By Theorem~\ref{theorem:mennicke}(1), it is enough to show that $\rho(\Br_{2g+1,1})$ contains every $\tau_{\vec v}^4$ with $\vec v$ primitive.  To this end, let $\vec v$ be a primitive element of $\Z^{2g}$.  By Proposition~\ref{proposition:presymmetricrep}, there is a pre-symmetric representative $c$ of $\vec v$ in $S_g^1$.

Let $\bar c$ denote the image of $c$ in $\D_{2g+1}$; the curve $\bar c$ is a simple closed curve surrounding an even number of marked points. Choose a simple closed curve $\bar d$ in $\D_{2g+1}$ so that $\bar c \cup \bar d$ form the boundary of an annulus containing the marked point $p_1$ and no other $p_i$ (if $\bar c$ surrounds exactly two marked points, then $\bar d$ is inessential).  

Clearly the element of $\PB_{2g+1}$ given by $T_{\bar c}^2 T_{\bar d}^{-2}$ lies in $\Br_{2g+1,1}$.  We claim that it maps to $\tau_{\vec v}^4$ under $\rho$.  The image of $T_{\bar c}^2$ in $\Mod(S_g^1)$ is $T_{c}^2T_{\iota(c)}^2$ and the image of the latter in $\Sp_{2g}(\Z)$ is $\tau_{[\vec c]}^2 \tau_{[\iota(\vec c)]}^2 = \tau_{\vec v}^4$ (here we choose an arbitrary orientation on $c$).  Next, since $\bar d$ surrounds an odd number of marked points, the image of $T_{\bar d}^2$ in $\Mod(S_g^1)$ is a Dehn twist about a separating curve and the image of the latter in $\Sp_{2g}(\Z)$ is trivial.  This gives the claim, and hence the statement that $\rho(\Br_{2g+1,1})$ contains $\Sp_{2g}(\Z)[4]$.

\medskip

We now proceed to compute the image of $\Br_{2g+1,k}$ in $\Sp_{2g}(\Z)[2]/\Sp_{2g}(\Z)[4]$.  Recall from Section~\ref{sec:level four mod 2} that $\Sp_{2g}(\Z)[2]/\Sp_{2g}(\Z)[4]$ is isomorphic to $(\Z/2)^{2g+1 \choose 2}$ and the 
coordinates of the resulting map $\PB_{2g+1} \to (\Z/2)^{2g+1 \choose 2}$ are given by the various maps $f_{ij}  : \PB_{2g+1} \to \PB_2/\PB_2^2$ 
obtained by forgetting all strands except the $i$th and $j$th and reducing modulo two.  

We claim that $f_{ij}(\Br_{2g+1,k})$ is nontrivial if and only if $\{1,\dots,k\} \subseteq \{i,j\}$.  Indeed, if $\{1,\dots,k\} \subseteq \{i,j\}$ then the pure braid $a_{ij}$ from Section~\ref{sec:sql} lies in $\Br_{2g+1,k}$ and has nontrivial image under $f_{ij}$. On the other hand if $\{1,\dots,k\} \not\subseteq \{i,j\}$ then $f_{ij}(\Br_{2g+1,k})$ is trivial by definition, giving the claim.  The description of the image of $\Br_{2g+1,1}$ follows immediately.

The theorem for the even-stranded braid group is proven in the same way, with $\PB_{2g+1}$, $\Br_{2g+1,k}$, and $\Sp_{2g}(\Z)[2]$ replaced with $\PB_{2g+2}$, $\Br_{2g+2,k}$, and $\left(\Sp_{2g}(\Z)[2]\right)_{\vec y_{g+1}}$ and with Theorem~\ref{theorem:mennicke}(1) replaced by Theorem~\ref{theorem:mennicke}(2).
\end{proof}

Theorem~\ref{theorem:mennicke} also holds in the case $g=1$ and $m=4$.  To see this, consider the action of $\Sp_2(\Z)[4]=\SL_2(\Z)[4]$ on the hyperbolic plane $\mathbb{H}^2$.  Using the Farey tessellation of $\mathbb{H}^2$, the quotient is easily seen to be a punctured sphere (it is an octahedron minus the vertices) and so the fundamental group is generated by loops around the punctures.  Since an octahedron has 4 triangles at each vertex, each corresponds to a 4th power of a transvection.

Hence, our proof of  Theorem~\ref{theorem:pointpushing} also applies in the case of $g=1$ and so $\rho(\Br_{3,1})$ contains $\SL_2(\Z)[4]$ and the image of $\rho(\Br_{3,1})$ in $\SL_2(\Z)[2]/\SL_2(\Z)[4]$ is $(\Z/2)^2$.  
The group $\BI_3$ is contained in the center of $\B_3$, which has trivial intersection with $\Br_{3,1}$.  It follows that there is a short exact sequence
\[
1 \to \SL_2(\Z)[4] \to \Br_{3,1} \to (\Z/2)^2.
\]
Specifically, $\Br_{3,1}$ is identified with the index 2 subgroup of $\SL_2(\Z)[2]$ consisting of matrices of the form $I+2A$ where $A$ is congruent to one of the following matrices modulo 2:
\[
\left(
\begin{array}{cc}
0 & 0 \\ 0 & 0
\end{array}
\right), \  
\left(
\begin{array}{cc}
0 & 1 \\ 0 & 0
\end{array}
\right), \  
\left(
\begin{array}{cc}
1 & 1 \\ 1 & 1
\end{array}
\right), \   \text{ or }\ 
\left(
\begin{array}{cc}
1 & 0 \\ 1 & 1
\end{array}
\right).
\]

\bibliographystyle{plain}
\bibliography{bncong}

\def\cprime{$'$} \def\cprime{$'$}
\begin{thebibliography}{10}

\bibitem{acampo}
Norbert A'Campo.
\newblock Tresses, monodromie et le groupe symplectique.
\newblock {\em Comment. Math. Helv.}, 54(2):318--327, 1979.

\bibitem{arnold}
V.~I. Arnol{\cprime}d.
\newblock A remark on the branching of hyperelliptic integrals as functions of
  the parameters.
\newblock {\em Funkcional. Anal. i Prilo\v zen.}, 2(3):1--3, 1968.

\bibitem{asada}
Mamoru Asada.
\newblock The faithfulness of the monodromy representations associated with
  certain families of algebraic curves.
\newblock {\em J. Pure Appl. Algebra}, 159(2-3):123--147, 2001.

\bibitem{assion}
Joachim Assion.
\newblock A proof of a theorem of {C}oxeter.
\newblock {\em C. R. Math. Rep. Acad. Sci. Canada}, 1(1):41--44, 1978/79.

\bibitem{bandboyland}
Gavin Band and Philip Boyland.
\newblock The {B}urau estimate for the entropy of a braid.
\newblock {\em Algebr. Geom. Topol.}, 7:1345--1378, 2007.

\bibitem{birmanhilden}
Joan~S. Birman and Hugh~M. Hilden.
\newblock On the mapping class groups of closed surfaces as covering spaces.
\newblock In {\em Advances in the theory of Riemann surfaces (Proc. Conf.,
  Stony Brook, N.Y., 1969)}, pages 81--115. Ann. of Math. Studies, No. 66.
  Princeton Univ. Press, Princeton, N.J., 1971.

\bibitem{bmp}
Tara Brendle, Dan Margalit, and Andrew Putman.
\newblock Generators for the hyperelliptic {T}orelli group and the kernel of
  the {B}urau representation at {$t=-1$}.
\newblock {\em Invent. Math.}, 200(1):263--310, 2015.

\bibitem{CP}
Thomas Church and Andrew Putman.
\newblock Generating the {J}ohnson filtration.
\newblock {\em Geom. Topol.}, 19(4):2217--2255, 2015.

\bibitem{cohenwu}
F.~R. Cohen and J.~Wu.
\newblock On braid groups and homotopy groups.
\newblock In {\em Groups, homotopy and configuration spaces}, volume~13 of {\em
  Geom. Topol. Monogr.}, pages 169--193. Geom. Topol. Publ., Coventry, 2008.

\bibitem{coxeter}
H.~S.~M. Coxeter.
\newblock Factor groups of the braid group.
\newblock {\em Proc. Fourth Canadian Math. Congress, Banff}, pages 95--122,
  1957.

\bibitem{DM}
P.~Deligne and G.~D. Mostow.
\newblock Monodromy of hypergeometric functions and nonlattice integral
  monodromy.
\newblock {\em Inst. Hautes \'Etudes Sci. Publ. Math.}, (63):5--89, 1986.

\bibitem{ddh}
Steven Diaz, Ron Donagi, and David Harbater.
\newblock Every curve is a {H}urwitz space.
\newblock {\em Duke Math. J.}, 59(3):737--746, 1989.

\bibitem{primer}
Benson Farb and Dan Margalit.
\newblock {\em A primer on mapping class groups}, volume~49 of {\em Princeton
  Mathematical Series}.
\newblock Princeton University Press, Princeton, NJ, 2012.

\bibitem{FunarKohno}
Louis Funar and Toshitake Kohno.
\newblock On {B}urau's representations at roots of unity.
\newblock {\em Geom. Dedicata}, 169:145--163, 2014.

\bibitem{GambaudoGhys}
Jean-Marc Gambaudo and {\'E}tienne Ghys.
\newblock Braids and signatures.
\newblock {\em Bull. Soc. Math. France}, 133(4):541--579, 2005.

\bibitem{hain}
Richard Hain.
\newblock Finiteness and {T}orelli spaces.
\newblock In {\em Problems on mapping class groups and related topics},
  volume~74 of {\em Proc. Sympos. Pure Math.}, pages 57--70. Amer. Math. Soc.,
  Providence, RI, 2006.

\bibitem{Humphries}
Stephen~P. Humphries.
\newblock Normal closures of powers of {D}ehn twists in mapping class groups.
\newblock {\em Glasgow Math. J.}, 34(3):313--317, 1992.

\bibitem{Humphries2}
Stephen~P. Humphries.
\newblock Subgroups of pure braid groups generated by powers of {D}ehn twists.
\newblock {\em Rocky Mountain J. Math.}, 37(3):801--828, 2007.

\bibitem{KhovanovSeidel}
Mikhail Khovanov and Paul Seidel.
\newblock Quivers, {F}loer cohomology, and braid group actions.
\newblock {\em J. Amer. Math. Soc.}, 15(1):203--271, 2002.

\bibitem{MagnusPeluso}
Wilhelm Magnus and Ada Peluso.
\newblock On a theorem of {V}. {I}. {A}rnol\cprime d.
\newblock {\em Comm. Pure Appl. Math.}, 22:683--692, 1969.

\bibitem{mcmullen}
Curtis~T. McMullen.
\newblock Braid groups and {H}odge theory.
\newblock {\em Math. Ann.}, 355(3):893--946, 2013.

\bibitem{mcr}
D.~B. McReynolds.
\newblock The congruence subgroup problem for pure braid groups: {T}hurston's
  proof.
\newblock {\em New York J. Math.}, 18:925--942, 2012.

\bibitem{Mennicke}
J.~Mennicke.
\newblock Zur {T}heorie der {S}iegelschen {M}odulgruppe.
\newblock {\em Math. Ann.}, 159:115--129, 1965.

\bibitem{morifuji}
Takayuki Morifuji.
\newblock On {M}eyer's function of hyperelliptic mapping class groups.
\newblock {\em J. Math. Soc. Japan}, 55(1):117--129, 2003.

\bibitem{tata1}
David Mumford.
\newblock {\em Tata lectures on theta. {I}}, volume~28 of {\em Progress in
  Mathematics}.
\newblock Birkh\"auser Boston, Inc., Boston, MA, 1983.
\newblock With the assistance of C. Musili, M. Nori, E. Previato and M.
  Stillman.

\bibitem{mumford}
David Mumford.
\newblock {\em Tata lectures on theta. {II}}, volume~43 of {\em Progress in
  Mathematics}.
\newblock Birkh\"auser Boston, Inc., Boston, MA, 1984.
\newblock Jacobian theta functions and differential equations, With the
  collaboration of C. Musili, M. Nori, E. Previato, M. Stillman and H. Umemura.

\bibitem{newmansmart}
M.~Newman and J.~R. Smart.
\newblock Symplectic modulary groups.
\newblock {\em Acta Arith}, 9:83--89, 1964.

\bibitem{smythe}
N.~F. Smythe.
\newblock The {B}urau representation of the braid group is pairwise free.
\newblock {\em Arch. Math. (Basel)}, 32(4):309--317, 1979.

\bibitem{Thurston}
William~P. Thurston.
\newblock Shapes of polyhedra and triangulations of the sphere.
\newblock In {\em The {E}pstein birthday schrift}, volume~1 of {\em Geom.
  Topol. Monogr.}, pages 511--549. Geom. Topol. Publ., Coventry, 1998.

\bibitem{Venk}
T.~N. Venkataramana.
\newblock Image of the {B}urau representation at {$d$}-th roots of unity.
\newblock {\em Ann. of Math. (2)}, 179(3):1041--1083, 2014.

\bibitem{bw}
Bronislaw Wajnryb.
\newblock On the monodromy group of plane curve singularities.
\newblock {\em Math. Ann.}, 246(2):141--154, 1979/80.

\bibitem{yu}
Jiu-Kang Yu.
\newblock {Toward a proof of the Cohen-Lenstra conjecture in the function field
  case}.
\newblock {\em Preprint}.

\end{thebibliography}


\begin{thebibliography}{1}

\bibitem{bmp}
Tara Brendle, Dan Margalit, and Andrew Putman.
\newblock Generators for the hyperelliptic {T}orelli group and the kernel of
  the {B}urau representation at {$t=-1$}.
\newblock {\em Invent. Math.}, 200(1):263--310, 2015.

\bibitem{bncong}
Tara~E. Brendle and Dan Margalit.
\newblock The level four braid group.
\newblock {\em J. Reine Angew. Math.}, 735:249--264, 2018.

\bibitem{sato}
Masatoshi Sato.
\newblock The abelianization of the level {$d$} mapping class group.
\newblock {\em J. Topol.}, 3(4):847--882, 2010.

\end{thebibliography}

\end{document}